\documentclass[10pt]{article}
\usepackage{amsmath,amssymb,amsthm}

\newtheorem{theorem}{Theorem}[section]
\newtheorem{proposition}[theorem]{Proposition}
\newtheorem{lemma}[theorem]{Lemma}
\newtheorem{corollary}[theorem]{Corollary}
\theoremstyle{definition}
\newtheorem{definition}[theorem]{Definition}

\theoremstyle{remark}
\newtheorem{remark}[theorem]{Remark}

\newcommand{\N}{\mathbb{N}}
\newcommand{\Z}{\mathbb{Z}}

\newcommand{\R}{\mathbb{R}}

\newcommand{\cH}{{\cal H}}
\newcommand{\cU}{{\cal U}}

\newcommand{\Zjs}{{\cal Z}}

\newcommand{\Ainfty}{{A^{\infty}}}
\newcommand{\Atinfty}{{A_{\rm t}^{\infty}}}
\newcommand{\Acentral}{{A_{\infty}}}
\newcommand{\Atcentral}{{A_{{\rm t} \infty}}}
\newcommand{\czero}{{c_0}}
\newcommand{\ctzero}{{c_{{\rm t} 0}}}
\newcommand{\deltae}{\partial_{\rm e}}
\newcommand{\tphi}{\widetilde{\varphi}}
\newcommand{\tI}{\widetilde{I}(A_0, A_1)}
\newcommand{\I}{I(A_0, A_1)}
\newcommand{\piomega}{\pi_{\omega}}
\newcommand{\Homega}{{\cal H}_{\omega}}

\DeclareMathOperator{\id}{id}
\DeclareMathOperator{\Imag}{Im}
\DeclareMathOperator{\Aff}{Aff}
\DeclareMathOperator{\Ad}{Ad}
\DeclareMathOperator{\bd}{bd}
\DeclareMathOperator{\conv}{conv}

\DeclareMathOperator{\supp}{supp}
\DeclareMathOperator{\Lip}{Lip}

\begin{document}
\title{Trace spaces of simple nuclear C$^*$-algebras with finite-dimensional extreme boundary}
\author{
Yasuhiko Sato \\
Kyoto University / University of Oregon\\
}
\date{ }

\maketitle

\begin{abstract}   
Let $A$ be a unital separable simple infinite-dimensional nuclear C$^*$-algebra with at least one tracial state. We prove that if the trace space of $A$ has compact finite-dimensional extreme boundary then there exist  unital embeddings of matrix algebras into a certain central sequence algebra of $A$ which is determined by the uniform topology on the trace space. As an application, it is shown  that if furthermore $A$ has strict comparison then $A$ absorbs the Jiang-Su algebra tensorially.
\end{abstract}

\section{Introduction}\label{Sec1}

\

\

\hrule width 6cm

\

\

The following is the main result of this paper.

\begin{theorem}\label{MainThm}
Let $A$ be a unital separable simple infinite-dimensional nuclear C$^*$-algebra with at least one tracial state. Suppose that the extreme boundary of $T(A)$ is a compact finite-dimensional space. Then for any $k\in \N$ there exists a unital embedding of the $k$ by $k$ matrix algebra into a variant of the central sequence algebra of $A$ defined by
\[ A'\cap \left(\l^{\infty} (\N, A) /\{ (a_n)_n \in \l^{\infty}(\N,A) : \lim_{n\rightarrow\infty} \max_{\tau\in T(A)} \tau(a_n^*a_n) =0\}\right).\]
\end{theorem}
Here, we denote by $T(A)$ the set of tracial states of $A$ which is called {\it  trace space} in \cite{Goo}, and we identify $A$ with the C$^*$-subalgebra of equivalence classes of constant sequences. 
As a main application of this theorem, we present the following result. Once we know the above theorem, the proof of this corollary can be obtained in the same way as the proof of \cite[Theorem 1.1]{MS2}.

\begin{corollary}\label{MainCor}
If $A$ and $T(A)$ satisfy the same conditions in the above theorem, then the following are equivalent:
\begin{enumerate}
\item $A\otimes\Zjs\cong A$.
\item $A$ has strict comparison.
\item Any completely positive map from $A$ to $A$ can be excised in small central sequences.
\item $A$ has property (SI).
\end{enumerate}

\end{corollary}

\section{A generalization of dimension drop algebras}\label{Sec2}

The unital embedding in the main theorem will be constracted as a factorization through a kind of dimension drop algebra $\Delta_{d,k}$ in Definition \ref{DefDelta}. The aim of this section is to introduce a generalization of dimension drop algebras and to obtain the universality of $\Delta_{d,k}$.

Throughout this section, we fix a separable infinite-dimensional Hilbert space $\cH$ and consider two unital separable nuclear C$^*$-algebras $A_i$, $i=0,1$. A completely positive contraction from $A_i$ to $B(\cH)$ is called {\it order zero} (or disjointness preserving) if it preserves orthogonality, \cite{WZ}, \cite{Wol}. Let $\Lambda_i$ be the set of all order zero completely positive contractions from $A_i$ to $B(\cH)$ for each $i=0,1,$ and let 
\begin{align*} \Lambda =\{ & (\varphi_0, \varphi_1)\in \Lambda_0\times\Lambda_1 \\ &:\quad
 [\varphi_0(a_0), \varphi_1(a_1)]=0\ {\rm for\ any\ }a_i \in A_{i}, i=0,1,\quad {\rm and}\quad \varphi_0(1_{A_0})+ \varphi_1(1_{A_1}) =1_{B(\cH)}\}. 
\end{align*}
We define two maps $\tphi_i:A_i \rightarrow \bigoplus_{\lambda\in\Lambda} \cH$, $i=0,1$, by 
\[ \tphi_i(a) = \bigoplus_{(\varphi_{0,\lambda}, \varphi_{1,\lambda})\in \Lambda}\varphi_{i,\lambda} (a), \quad a\in A_i,\ i=0,1.\]
Then it follows that $\tphi_i$, $i=0,1$, are also order zero completely positive contractions such that $[\tphi_0(a_0),$ $ \tphi_1(a_1)]=0$ for any $a_i\in A_i$, $i=0,1$, and $\tphi_0(1_{A_0}) + \tphi_1(1_{A_1}) = 1$. We denote by $\tI$ the C$^*$-subalgebra of $B(\bigoplus_{\Lambda} \cH)$ generated by $\tphi_0(A_0)\cup \tphi_1(A_1)$. It is not hard to check that $\tI$ satisfies the universal property for the relations in $\Lambda$.

Letting 
\[ \I= \{ f\in C([0,1])\otimes A_0 \otimes A_1\ : \  f(0)\in A_0\otimes 1_{A_1},\ f(1)\in 1_{A_0}\otimes A_1\}, \]
we have the following lemma. The argument in the proof is a slight generalization of  \cite[Proposition 7.3]{JS}, and \cite[Proposition 2.5]{RW}.
\begin{lemma}\label{LemUniv}
\
\begin{enumerate}
\item $\tI$ is isomorphic to $\I$. 
\item If both $A_0$ and $A_1$ contain $M_k$ unitally then $I(A_0,A_1)$ also contains $M_k$ unitally.
\end{enumerate}
\end{lemma}
\begin{proof}[Proof of (i).]
First, we define $z\in I(A_0, A_1)$ and $\varphi_i: A_i \rightarrow I(A_0, A_1)$, $i=0, 1$ by $z(t)=t 1_{A_0}\otimes 1_{A_1} \in A_0\otimes A_1$ for $t\in [0,1]$, and
\[\varphi_0(a)= (1-z)a\otimes 1_{A_1}\ {\rm for\ } a\in A_0\quad {\rm and\quad } \varphi_1 (a) =z 1_{A_0}\otimes a_1\ {\rm for\ } a\in A_1.\] It follows that $\varphi_i: A_i \rightarrow I(A_0,A_1)$, $i=0,1$, are order zero completely positive contractions satisfying $[\varphi_0 (a_0), \varphi_1(a_1)]=0$ for $a_i\in A_i$, $i=0,1$, and $\varphi_0(1_{A_0})+\varphi_1(1_{A_1}) =1$. Taking a unital faithful representation of $I(A_0, A_1)$ on $\cH$ we regard $I(A_0, A_1)$ as a unital C$^*$-subalgebra of $B(\cH)$, thus this means that $(\varphi_0, \varphi_1)\in \Lambda$. Let $\Phi$ $:\tI \rightarrow \I$ be the $*$-homomorphism defined by the canonical projection $B(\bigoplus_{\lambda\in\Lambda}\cH)\rightarrow B(\cH)$ determined by $(\varphi_0, \varphi_1)\in\Lambda$. From these definitions we have 
\[\Phi\circ\tphi_i (a)= \varphi_i (a) \quad {\rm for\ } a\in A_i,\ i=0,1.\]

Now we have $\Imag(\Phi)(0)= A_0\otimes 1_{A_1}$, $\Imag(\Phi)(1) = 1_{A_0}\otimes A_1$, and $\Imag(\Phi)(t)=A_0\otimes A_1$ for $t\in (0,1)$, where we set $B(t)=\{f(t): f\in B\}$ for a C$^*$-subalgebra $B\subset \I$ and $t\in [0,1]$. Then by partitions of unity in $C^*(\{z\})\cong C([0,1])$, the standard argument implies that $\Phi$ is surjective.

Because $\tphi_i$, $i=0,1$ preserve orthogonality, by  \cite[Theorem 2.3]{WZ} or \cite[Theorem 2.3]{Wol}, we obtain the $*$-homomorphism $\pi_i: A_i \rightarrow \tI '' \subset B(\bigoplus_{\lambda\in\Lambda} \cH)$, $i=0,1$, such that 
\[[\pi_i(a), \tphi_i(1_{A_i})]=0 \quad {\rm and}\quad \tphi_i (a) =\tphi_i(1_{A_i})\cdot \pi_i(a)\quad { \rm for\ } a\in A_i, i=0,1.\] Indeed, these $\pi_i(a)$ were constructed as the strong limit of $f_n(\tphi_i(1_{A_i}))\cdot\tphi_i(a)$ for some positive functions $f_n\in C([0,\infty))$, $n\in\N$. Then we also have $[\pi_0(a_0), \pi_1(a_1)] =0$ for $a_i\in A_i$, $i=0,1$. Set $\rho =\pi_0\otimes \pi_1$ $:A_0\otimes A_1\rightarrow \tI''$ and set $\pi_t :\I\rightarrow A_0\otimes A_1$ as $\pi_t(f)=f(t)$ for $t\in [0,1]$ and $f\in \I$.

For a pure state $\omega$ of $\tI$, let $(\piomega,\Homega)$ be the GNS-representation associated with $\omega$. Since $\tphi_i (1_{A_i})$, $i=0,1$, are in the centre of $\tI$ we obtain scalar values $t_{\omega, i}\in [0,1]$, $i=0,1$,  such that $t_{\omega, i} 1_{B(\Homega)}= \piomega(\tphi_i(1_{A_i}))$ and $t_{\omega,0}+t_{\omega,1}=1$. Because for any nonzero $x\in \tI_+$ there exists a pure state $\omega_x$ such that $\pi_{\omega_x}(x) > 0$, in order to show that $\Phi$ is injective,  it suffices to show that 
\[ \piomega (x) = \piomega\circ\rho\circ\pi_{t_{\omega,1}}\circ\Phi (x), \]
for any $x\in\tI$ and pure state $\omega$ of $\tI$.
When $x=\tphi_i(a_i)$ for $a_i\in A_i$ we have 
\[\piomega(x) = \piomega\circ \pi_i (t_{\omega,i}a_i) 
=\piomega \circ \rho\circ \pi_{t_{\omega,1}}\circ \varphi_i (a_i) 
=\piomega\circ \rho \circ \pi_{t_{\omega,1}}\circ \Phi (x). \]
Since $\tI$ is generated by $\tphi_0 (A_0)\cup \tphi_1(A_1)$ and $\piomega\circ\rho\circ\pi_{t_{\omega,1}}\circ\Phi$ is a $*$-homomorphism we conclude that $\piomega =\piomega\circ\rho\circ\pi_{t_{\omega,1}}\circ\Phi$.
\end{proof}
\begin{proof}[Proof of (ii).]
Let $\Phi_i : M_k \rightarrow A_i$, $i=0,1$, be given unital embeddings.
We denote by $u\in M_k\otimes M_k$ the self-adjoint unitary such that $\Ad u(x\otimes y) =y\otimes x$ for $x, y\in M_k$, and we let $\widetilde{u} \in C([0,1]) \otimes M_k \otimes M_k$ be a path of unitaries such that $\widetilde{u} (0) =1$, $\widetilde{u} (1) =u$. Then the unital $*$-homomorphism $\Phi : M_k \rightarrow C([0,1])\otimes A_0\otimes A_1$ defined by 
\[\Phi (x) = \id_{C([0,1])}\otimes\Phi_0\otimes\Phi_1 \circ \Ad\widetilde{u} (1_{C([0,1])}\otimes x\otimes 1_k), \]
satisfies $\Imag (\Phi ) \subset \I$.
\end{proof}

\begin{definition}\label{DefDelta}
For $d, k\in \N$, we inductively define a unital C$^*$-algebra $\Delta_{d,k}$ by
\[ \Delta_{d,k} = I(\Delta_{d-1, k}, M_k),\quad \Delta_{0,k} =M_k. \]
 In the proof of the main theorem, we shall see that this $\Delta_{d,k}$ can be embedded into the quotient algebra in Theorem \ref{MainThm}, and this $d$ corresponds to the covering dimension of $\deltae(T(A))$.
\end{definition}

\begin{corollary}\label{CorDelta}
\ 
\begin{enumerate}
\item $\Delta_{d,k}$ is isomorphic to the universal C$^*$-algebra on generators
$\{e_{l,i} : l=0,1,...,d, i=1,2,...,k\}$ satisfying taht:
\begin{align*}  
&\sum_{l=0}^d\sum_{i=1}^k e_{l,i}^* e_{l,i} =1, \quad e_{l,i}e_{l,j}^* =\delta_{i,j} e_{l,1}^2,\quad l=0,1,...,d,\quad i,j=1,2,...,k,\\ 
& [e_{l,i}, e_{m,j}]=0,\quad l\neq m,\quad i, j = 1,2,...,k.
\end{align*}
\item $\Delta_{d,k}$ contains $M_k$ unitally. 
\end{enumerate}
\end{corollary}
\begin{proof}
(ii) is straightforward from (ii) of Lemma \ref{LemUniv}. 

For $d, k\in \N$, let $\cU_{d,k}$ be the universal C$^*$-algebra in the corollary and let $\{e_{l,i}^{(d)}: l=0,1,...,d, i=1,2,...,k\} \subset \cU_{d,k}$ be a set of generators satisfying the relations in (i). Because of the universal property of $\cU_{d,k}$ and $\cU_{0,k}\cong M_k$, in a similar fashon to the proof of  \cite[Proposition 4.1.1]{WinCovDim} we can obtain order zero completely positive contractions $\varphi_0: \cU_{d-1,k} \rightarrow \cU_{d,k}$ and $\varphi_1 :M_k\rightarrow\cU_{d,k}$ such that $\varphi_0(e_{l,i}^{(d-1)}) =e_{l,i}^{(d)}$ for $l=0,1,...,d-1$, $i=1,2,...,k$, $\varphi_1 (e_{0,i}^{(0)}) =e_{d,i}^{(d)}$ for $i=1,2,...,k$, and $(\varphi_0, \varphi_1)\in \Lambda$. Set $\tphi_0 :  \cU_{d-1,k} \rightarrow \widetilde{I}(\cU_{d-1,k}, M_k)$ and $\tphi_1 : M_k \rightarrow \widetilde{I}(\cU_{d-1,k}, M_k)$ as the above.  By the universality of $\widetilde{I}(\cU_{d-1,k}, M_k)$ there exists a $*$-homomorphism $\Phi$ from $\widetilde{I}(\cU_{d-1,k}, M_k)$ onto $\cU_{d,k}$ such that $\Phi(\tphi_i(x))=\varphi_i(x)$, $i=0,1$. On the otherhand $\{\tphi_0(e_{l,i}^{(d-1)}) \}_{l,i} \cup \{ \tphi_1 (e_{0,i}^{(0)})\}_i$ satisfies the relations for $\cU_{d,k}$, then we have a unital $*$-homomorphism $\Psi: \cU_{d,k} \rightarrow \widetilde{I} (\cU_{d-1,k}, M_k)$ wich satisfies $\Psi (e_{l,i}^{(d)}) = \tphi_0(e_{l,i}^{(d-1)})$ and $\Psi (e_{d,i}^{(d)}) = \tphi_1( e_{0,i}^{(0)})$ for $l=0,1,...,d-1$, $i=1,2,...,k$. These $\Phi$ and $\Psi$ imply that $\cU_{d,k} \cong \widetilde{I} (\cU_{d-1,k}, M_k)$. 

When $d=0$ the statement is trivial. Assume that we have seen the statement for $d-1$, i.e., $\Delta_{d-1,k} \cong \cU_{d-1,k}$. By Lemma \ref{LemUniv} it follows that $\Delta_{d,k} =I(\Delta_{d-1,k}, M_k) \cong I(\cU_{d-1,k}, M_k) \cong \widetilde{I}(\cU_{d-1,k}, M_k)$. Then (ii) follows from the induction.   
\end{proof}

\section{ Central sequence algebras and tracial states}\label{Sec3}

In this section, we recall the notion of central sequence algebras and prove Corollary \ref{CorCentSeq}, which is one of the fundamental lemmas to study central sequences and tracial states. In what follows we let $A$ be a separable simple C$^*$-algebra with at least one tracial state. Define
\begin{align*}
\| a\|_{\tau} = \tau (a^*a)^{1/2},\quad \tau\in T(A),\ a\in A,\quad 
\| a\|_2 = \sup_{\tau\in T(A)} \|a\|_{\tau}, \quad  a\in A,
\end{align*}
\begin{align*}
 \czero &= \{ (a_n)_n\in \l ^{\infty} (\N, A) : \lim_{n\to\infty} \|a_n\|=0\},  &\ctzero &= \{ (a_n)_n\in \l^{\infty} (\N, A) : \lim_{n\to \infty} \|a_n\|_2 =0 \},\\
 \Ainfty &=\l ^{\infty} (\N, A) / \czero,& \Atinfty&=\l^{\infty} (\N, A)/ \ctzero,
\end{align*}
(see also \cite[Section 2]{MS3}).
Since $\czero\subset\ctzero$, we can regard $\Atinfty$ as the quotient algebra of $\Ainfty$. We identify $A$ with the C$^*$-subalgebra of $\Ainfty$ (resp. $\Atinfty$) consisting of equivalence classes of constant sequences. We let
\[ \Acentral = \Ainfty\cap A',\quad \quad  \Atcentral= \Atinfty \cap A'.\]

This $\Acentral$ is called the {\it central sequence algebra} of $A$ for C$^*$-algebras. A sequence $(a_n)_n\in \l^{\infty}(\N, A)$ is called a {\it central sequence} if $\|[a_n, x]\|\to 0$ as $n\to \infty$ for any $x\in A$. A central sequence is a representative of an element in $\Acentral$. In \cite{Bla}, B. Blackadar introduced the notion of strict comparison (for projections) by taking into consideration the uniform topology on the trace space of C$^*$-algebras. M. R\o rdam adapted that  strict comparison in order to apply Goodearl-Handelman's Hahn-Banach type theorem \cite{GH}, and he proved that $\Zjs$-absorption implies strict comparison in \cite{RorUHF}, \cite{Ror}. For this reason, we define the above 2-norm $\|\cdot \|_2$ by the uniformness on $T(A)$.

\begin{proposition}\label{PropCentSeq}
Let $A$ be a unital nuclear C$^*$-algebra. Then for any finite subset $F$ of $A$ and $\varepsilon >0$ there exist unitaries $u_1, u_2,...,u_N$ of $A$ such that 
\[ \left\| \left[\frac{1}{N}\sum_{i=1}^N \Ad u_i (a),\ b \right] \right\| <\varepsilon,\quad {\rm for\ all\ } a, b\in F.\]
\end{proposition}
\begin{proof}
A. Connes showed that any injective von Neumann algebra on a separable Hilbert space is approximately finite dimensional (AFD) \cite{Con}. Based on this result, G. A. Elliott generalized it for any von Neumann algebra \cite{Ell}. E. G. Effros and C. Lance showed that $A^{**}$ is injective if $A$ is nuclear \cite{EL}. Combining their results we can see that $A^{**}$ is AFD, i.e., for a finite subset $F$ of $A^{**}$, a finite subset $G$ of $A^*_+$, and $\varepsilon >0$, there exist a finite dimensional subalgebra $B$ of $A^{**}$ and $b_f \in B$, for $f\in F$ such that
\[ \|f - b_f \|_{\varphi}^{\sharp} < \varepsilon /6\quad {\rm for\ all \ } f\in F, \varphi\in G,\]
here we define $\|x\|_{\varphi}^{\sharp} = \sqrt{(x^*x + xx^*)/2}$ for $x\in A^{**}$, $\varphi\in A^*_+$ which induces the strong$*$ topology. By adding $1_{A^{**}}-1_B$, if necessary, we may assume that any unitary of $B$ is a unitary in $A$. To show the claim, we may assume that any $f\in F$ is self-adjoint and $\|f\|\leq 1$. Thus we also obtain $b_f\in B$ as a self-adjoint element of $B$  satisfying the above condition for $f\in F$. And by the argument in the proof of Kaplansky's density theorem, we can obtain  $b_f\in B$ as a self-adjoint element of $B$ with $\|b_f\| \leq 1$. Actually, since the continuous function $g(t)=\max \{\min \{t, 1\}, -1\}$, $t\in \R$ is a strongly continuous function (see \cite[Proposition 2.3.2]{Ped} for example) it suffices to consider $g(b_f)$.

Since the convex hull of the unitary group is norm dence in the unit ball for any C$^*$-algebra \cite{RD}, there exists a finite subset $F_B$ of unitaries in $B$ such that 
\[ \min\{ \|b_f - x\| : x\in \conv (F_B)\}  < \varepsilon/8. \]
Since $B$ is finite-dimensional, for $\varepsilon>0$ there exists another finite subset $D$ of unitaries in $B$ and permutations $\sigma_b$, $b\in F_B$, of $D$ such that 
\[ \| d\cdot b - \sigma_b(d) \| < \varepsilon/8\quad {\rm for\ } b\in F_B, d\in D.\]
The precise argument of this technique is written in \cite[Lemma 3.6]{KS}. Set $N=|D|\in \N$. For any contraction $a\in A$ and $b\in F_B$ we have 
\begin{align*}
&\|[\frac{1}{N}\sum_{d\in D} d^* a d, b]\| < \frac{1}{N} \| \sum_{d\in D} d^* a \sigma_b(d) - \sigma_b^{-1}(d)^* a d \| + \varepsilon/4 =\varepsilon /4. \\
\intertext{ From this, it follows that for any contraction $a\in A$}
& \|[\frac{1}{N} \sum_{d\in D} d^* ad, x]\| < \varepsilon/4 \quad {\rm for\ } x\in \conv (F_B), \\ 
& \|[\frac{1}{N} \sum_{d\in D} d^*a d, b_f] \| < \varepsilon/2\quad {\rm for\ } f\in F.\\  
\end{align*}
Then, for $\varphi\in G$, $f\in F$, and $a\in A$ with $\|a\|\leq 1$ we have 
\begin{align*}
&|\varphi([\frac{1}{N} \sum_{d\in D} d^* a d, f] )| \\
&\leq|\varphi([\frac{1}{N} \sum_d d^* a d, b_f])|+|\varphi(\frac{1}{N}\sum_d d^*ad(f-b_f))| + |\varphi(\frac{1}{N}\sum_d (f-b_f)d^* ad )| \\ 
&\leq \varepsilon /2 + 2\sqrt{2}\| f -b_f \|_{\varphi}^{\sharp} < \varepsilon.
\end{align*}

By Kaplansky's density theorem there exists a net  $e_{\lambda}$, $\lambda\in\Lambda$ of unitaries in $A$ such that $e_{\lambda}$ converges to $d\in A^{**}$ strongly$*$. Then $e_{\lambda}^* a e_{\lambda}$ converges to $d^* a d$ strongly for any $a\in A$. Thus, for finite subsets $F\subset A$ and $G\subset A^*_+$ we obtain unitaries $e_d \in A$, $d\in D$, satisfying that: for all $a, b\in F$
\begin{align*}
|&\varphi([\frac{1}{N} \sum_{d\in D} e_d^* a e_d, b] )| \\
&\leq |\varphi([\frac{1}{N} \sum_d d^* a d, b ])| +| \frac{1}{N} \sum_d \varphi(({e_d}^*a e_d - d^* a d)b)| + |\frac{1}{N} \sum_d \varphi(b({e_d}^*a e_d - d^* ad ))|< \varepsilon.
\end{align*}

Let $B(A,A)$ be the Banach algebra of all bounded operators from $A$ to $A$, $C_0$ the convex hull of $\{\Ad u: u {\rm \ is\ a\ unitary\ in\ }A\}$ in $B(A,A)$, and let
\[ C= \{([\Phi(a), b])_{(a,b)}\in \bigoplus_{F\times F} A : \Phi\in C_0\} .\]
Note that $C$ is also a convex set in the C$^*$-algebra $\bigoplus_{F\times F} A$. The above argument means that the weak closure of $C$ $\subset \bigoplus_{F\times F}A$ contains $0 \in \bigoplus_{F\times F} A$. Actually, what we do is as follows: For any $\varphi \in A^*$, by the Jordan decomposition, we have $\varphi_j\in A^*_+$, $j=1,2,3,4$ such that $\varphi=\sum_{j=1}^4(\sqrt{-1})^j \varphi_j$. Identifying $(\bigoplus_{F\times F}A)^*$ with $\bigoplus_{F\times F} A^*$, we let $\widetilde{G}$ be a finite subset of $\bigoplus_{F\times F} A^*$. Apply the above argument to $F\subset A$, $G=\{ \varphi_{{(a,b)}\ j}: \varphi\in \widetilde{G}, a, b \in F, j=1,2,3,4\}$, and $\varepsilon / (2|F|)^2$, then there exists $\Phi\in C_0$ such that 
\[| \varphi_{(a,b)\ j}([ \Phi (a), b])| < \varepsilon/(2|F|)^2,\quad {\rm for\ }\varphi\in \widetilde{G}, a, b\in F, j=1,2,3,4.\]
This implies 
\[|\varphi(([\Phi (a), b])_{(a,b)})|=|\sum_{a,b\in F} \varphi_{(a,b)} ([\Phi (a), b])| \leq \sum_{a,b\in F} \sum_{j=1}^4 |\varphi_{(a,b)\ j} ([\Phi(a), b])| <\varepsilon \quad {\rm for\ } \varphi \in \widetilde{G}. \]

Therefore the Hahn-Banach theorem shows that $0$ is contained in the norm closure of $C$. This means for any $\varepsilon>0$ there exists $\Phi\in C_0$ such that $\|[\Phi (a), b]\| < \varepsilon$ for $a, b\in F$. 
\end{proof}

\begin{remark}
(In the proof of the above proposition, the required $\Phi_0 \in C_0$ is heavily depend on double-dealing of a finite subset $F$ of $A$. So this proposition is much weaker than the strong amenability defined by B. E. Johnson \cite{JohText}.)
\end{remark}

\begin{corollary}\label{CorCentSeq}
Let $A$ be a unital separable nuclear C$^*$-algebra with at least one tracial state. Then for any $a\in A$ there exists a central sequence $a_n \in A$, $n\in \N$ such that $\|a_n\| \leq \| a\|$ and 
\[\tau(a) = \tau (a_n)\quad {\rm for\ any\ } \tau \in T(A) {\rm\ and\ } n\in \N.\] 
\end{corollary}
\begin{proof}
Since $A$ is separable there exists an increasing sequence $F_n$, $n\in \N$ of finite subsets of $A$ such that $\bigcup_{n\in \N} F_n \subset A$ is dense in the operator norm topology. Let $\varepsilon_n > 0$, $n\in \N$, be a decreasing sequence which converges to 0. Applying Proposition \ref{PropCentSeq} to $\{a \}\cup F_n$
and $\varepsilon_n > 0$ we obtain unitaries $u_{n,i}$, $i=1,2,...,N_n$ in $A$ satisfying the condition in Proposition \ref{PropCentSeq}. We define 
\[ a_n=\frac{1}{N_n} \sum_{i=1}^{N_n} \Ad u_{n,i} (a),\quad n\in \N.\] These $a_n\in A$, $n\in \N$ satisfy the required conditions. 
\end{proof}

\section{Orthogonality on the compact extreme boundary}\label{Sec4}

In this section, we give a simple main technical tool Lemma \ref{TechLem}
concerning multiplicativity and orthogonality on the compact extreme boundary $\deltae(T(A))$.

Recently, by using the next proposition M. Dadarlat and A. S. Toms investigated the dimension functions on the compact finite-dimensional extreme boundary of trace space, ( in the proof of \cite[Lemma 4.4]{DT}). This result was essentially based on the works by D. A. Edwards \cite{Ed}, J. Cuntz and G. K. Pedersen \cite{CP}, and H. Lin \cite{Lin}. The starting point of our proof of Lemma \ref{TechLem} is this proposition.

\begin{proposition}\label{PropDeltae}
Let $A$ be a unital separable simple infinite-dimensional C$^*$-algebra with at least one tracial state. Suppose that $\deltae (T(A))$ is compact. Then for any positive function $f\in C(\deltae (T(A))$ there exists a sequence $a_n$, $n\in \N$ of positive elements in $A$ such that
\[\lim_{n\to\infty}\max_{\tau\in \deltae (T(A))} |\tau(a_n) -f(\tau)| =0 \quad {\rm and}\quad \|a_n\| \leq \|f\|\ {\rm for\ }n\in \N.\]
\end{proposition}
\begin{proof}
First, in the study of Choquet simplex it turns out that if a Choquet simplex $K$ has compact extreme boundary $\deltae(K)$ then the natural restriction map from the continuous affine functions $\Aff(K)$ to $C(\deltae (K))$ is isometric isomorphism  \cite{Ed}, see also \cite[Corollary 11.21]{Goo}. Then for a given positive continuous function $f$ and $\varepsilon >0$, there exists a continuous affine function $F_{f,\varepsilon} \in \Aff(T(A))$ such that $F_{f,\varepsilon}|_{\deltae(T(A))}=f +\varepsilon$.

In \cite{CP}, we have seen that $T(A)\cong (A/A_0)^*$, where $A_0$ was defined as the linear space generated by $\{a^*a- aa^* : a\in A\}$.  Since $F_{f,\varepsilon}$ is a w$*$-continuous function of $(A/A_0)^{**}$ it follows that $F_{f,\varepsilon}\in A/A_0$. Then there exists a representative $a_{f,\varepsilon}'\in A$ of $F_{f,\varepsilon}$ such that $\| a_{f,\varepsilon}'\| =\|F_{f,\varepsilon}\|=\|f\|+\varepsilon$. By $\tau(a_{f,\varepsilon}')>0$ for any $\tau \in T(A)$ and \cite[Corollary 6.4]{CP}, we obtain $a_{f,\varepsilon}\in A_+$ such that $\tau(a_{f, \varepsilon})=\tau (a_{f, \varepsilon}')=f(\tau)+\varepsilon$ for any $\tau\in T(A)$. From \cite[Theorem 2.9]{CP}, we may obtain the condition $\|a_{f, \varepsilon}\| < \|f\|+2\varepsilon$, the same argument appears in the proof of \cite[Theorem 9.3]{Lin}. By taking a decreasing sequence $\varepsilon_n >0$, $n\in\N$ which converges to 0 and taking a small perturbation of $a_{f, \varepsilon_n}$ we obtain $a_n \in A_+$, $n\in\N$ satisfying the desired conditions.   
\end{proof}

In the following lemma, (i) is a variant of \cite[Lemma 4.6]{MS1} and (ii) is a version of \cite[Lemma 3.2]{MS2} for the uniform topology on $\deltae(T(A))$.
\begin{lemma}\label{TechLem}
Let $A$ be a unital separable simple infinite-dimensional C$^*$-algebra. Suppose that $\deltae (T(A))$ is compact. Then the following hold:
\begin{enumerate}
\item For any central sequence $(f_n)_n\in \Acentral $ and $a\in A$, it follows that
\[ \lim_{n\to\infty} \max_{\tau\in\deltae (T(A))} |\tau (f_na)-\tau(f_n)\tau(a)| =0.\]
\item Moreover, if $A$ is nuclear, for mutually orthogonal positive functions $f_i \in C(\deltae (T(A)))$, $i=1,2,...,N$ there exist central sequences $(a_{i,n})_n$, $i=1,2,...,N$ of positive elements in $A$ such that 
\[\lim_{n\to\infty} \max_{\tau\in \deltae(T(A))} |\tau(a_{i,n})-f_i (\tau)| =0\quad {\rm for\ }i=1,2,...,N,\quad {\rm and\ } \lim_{n\to\infty} \|a_{i,n}a_{j,n}\| =0  {\rm \ for\ } i\neq j. \]
\end{enumerate}
\end{lemma}
\begin{proof}[Proof of (i)]
Since $A$ is separable, $\deltae (T(A))$ is a compact metric space. In what follows we denote by $d$ a metric of $\deltae (T(A))$ and let $B(\tau, \varepsilon)=\{ \sigma\in \deltae (T(A)): d(\sigma, \tau)\leq\varepsilon\}$ for $\tau\in \deltae (T(A))$ and $\varepsilon >0$. To show (i), without loss of generality, we may assume that $\|a\|\leq 1$, $\|f_n\|\leq 1$ for $n\in\N$, and the Lipschitz constant of $\deltae (T(A))\ni \tau \mapsto \tau(a)$ is less than one.

Since $\deltae (T(A))$ is a compact metric space, for $\varepsilon >0$ and $\tau_0\in \deltae (T(A))$ there exists a positive contraction $f_0\in C(\deltae (T(A)))$ such that $f_0|_{B(\tau_0, \varepsilon/16)}=1$ and the diameter of $\supp (f)$ is less than $\varepsilon /4$, and there exists a partition of unity $\{f_i'\}_{i=1}^N\subset C(\deltae (T(A)))$ such that the diameter of $\supp (f_i')$ are less than $\varepsilon /4$. Set $f_i = f_i'(1-f_0)$, $i=1,2,...,N$. Thus we see that $\{f_i\}_{i=0}^N$ is also a partition of unity. By changing subscript $i=1,2,...,N$, if necessary, we may assume that $f_i \neq 0$ for all $i=1,2,...,N$. Set $\sigma_0=\tau_0$ and $\sigma_i\in f_i^{-1} ((0,1])$ for each $i=1,2,...,N$.

By Proposition \ref{PropDeltae}, there exist sequences $a_{i,n}\in A$, $n\in\N$, $i=0,1,...,N$ of positive elements such that $\max_{\tau\in\deltae (T(A))}|\tau (a_{i,n}) -f_i(\tau)|\to 0$ as $n\to\infty$ and $\|a_{i,n}\|\leq 1$. Then, by taking a large $n\in \N$ we obtain positive contractions $a_i\in A$, $i=0,1,...,N$ such that 
\begin{align*}
&\max_{\tau\in \deltae (T(A))} |\tau (a_i) - f_i(\tau)| < \varepsilon/(8(N+1))\quad {\rm for\ }i=0,1,...,N, \\ 
&\max_{\tau\in\deltae (T(A))} |\tau (a - \sum_{i=0}^N \sigma_i(a)a_i) |\leq \max_{\tau} \sum_i|\tau(a)- \sigma_i (a)|f_i(\tau) +|\sigma_i(a)|\cdot|f_i(\tau) -\tau(a_i)| \\
&\leq \Lip (a)\varepsilon/4 < \varepsilon/4. 
\end{align*}
Set a positive element $\displaystyle a'=\sum_{i=0}^N \sigma_i(a) a_i \in A$. From the Krein-Milman theorem it follows that 
\[ \max_{\tau\in T(A)} |\tau (a-a')| < \varepsilon/4.\]
Because of Cuntz-Pedersen's theorem \cite{CP} (see also the proof of \cite[Lemma 4.6]{MS1}) we can obtain $u_j \in A$, $j=1,2,...,M$ such that
\[ \|a- \sum_{j=1}^M u_j^*u_j\|<\varepsilon/4, \quad \|a' -\sum_{j=1}^M u_ju_j^* \|< \varepsilon/4.\]
Since $(f_n)_n$ is a central sequence there exists $n_{\tau_0}\in \N$ such that for $n\geq n_{\tau_0}$ and $\sigma\in T(A)$,
\begin{align*}
\sigma(af_n)&\approx_{\varepsilon/4} \sigma(\sum_{j=1}^M u_j^*u_j f_n) \approx_{\varepsilon/8} \sigma (\sum_j u_ju_j^* f_n)\approx_{\varepsilon/4} \sigma (a'f_n)\\ 
&=\sum_{i=0}^N \sigma_i(a)\sigma(a_i f_n). 
\end{align*}
By the definition of $f_0$ we have $f_i (\sigma) =\delta_{i,0}$ for $\sigma\in B(\tau_0, \varepsilon/16)$ and $i=0,1,...,N$, then for $\sigma\in B(\tau_0,\varepsilon/16)$ and $n\in\N$ it follows that 
\begin{align*}
&|\sigma(a_if_n)-\delta_{i,0}\sigma(f_n)| = |\sigma((a_i-\delta_{i,0})f_n)| \\
&=|\sigma((-1)^{\delta_{i,0}}(a_i-\delta_{i,0})f_n)| \leq | \sigma(a_i-\delta_{i,0})|\cdot\|f_n\| \leq \varepsilon/(8(N+1)).
\end{align*}
 From this, for $n\geq n_{\tau_0}$ and $\sigma\in B(\tau_0, \varepsilon/16)$ we have
\[ |\sigma(af_n) -\sigma(a)\sigma(f_n) | 
\leq |\sigma (af_n) -\tau_0(a)\sigma(f_n) | + |\tau_0(a) -\sigma (a)|< 3\varepsilon/4 + \varepsilon/16 < \varepsilon. \]

By the compactness of $\deltae (T(A))$, there exist $\tau_1,\tau_2,...,\tau_L\in\deltae (T(A))$ such taht $\displaystyle \bigcup_{i=1}^L B(\tau_i, \varepsilon/16) \supset \deltae (T(A))$. Taking $n_{\tau_i}\in \N$, $i=1,2,...,L$ like $n_{\tau_0}$ in the above argument, we define $\displaystyle m=\max_{i=1,2,...,L} n_{\tau_i}$. Then if $n\geq m$ we conclude that 
\[\max_{\tau\in \deltae (T(A))}|\tau(af_n) -\tau (a) \tau(f_n) | < \varepsilon.\]\end{proof}
\begin{proof}[Proof of (ii).]
By Proposition \ref{PropDeltae}, we obtain sequences of positive elements $b_{i,n}'\in A$, $i=1,2,...,N$, $n\in\N$ such taht $\|b_{i,n}'\| \leq \|f_i\|$ and $\displaystyle \lim_{n\to\infty}\max_{\tau\in \deltae(T(A))}| \tau (b_{i,n}') - f_i(\tau)| =0$. 
By Corollary \ref{CorCentSeq} there exist central sequences $b_{i,n,m}'$, $m\in\N$ of  positive  elements in $A$ such that $\| b_{i,n,m}\| \leq \| b_{i,n}'\|$ and $\tau (b_{i,n,m}) =\tau (b_{i,n}')$ for $i=1,2,...,N$, $n, m\in\N$, and $\tau\in T(A)$. Since $A$ is separable, we can take a subsequence $m_n\in\N$, $n\in\N$ such that $(b_{i,n,m_n})_n\in \Acentral$ for $i=1,2,...,N$. Let $b_{i,n}= b_{i,n,m_n}$, $i=1,2,...,N$, $n\in\N$. Now we have sequences of positive elements $b_{i,n}\in A$, such that $\|b_{i,n}\|\leq \|f_i\|$, 
\[(b_{i,n})_n\in\Acentral,\quad {\rm and\ } \lim_{n\to \infty}\max_{\tau\in \deltae (T(A))}|\tau (b_{i,n}) -f_i(\tau)| =0.\]

By (i) and taking subsequences of $(b_{i,n})_n$ inductively we may assume that 
\[ \lim_{n\to\infty}\max_{\tau\in\deltae(T(A))}|\tau(b_{i,n}b_{j,n}) -\tau(b_{i,n})\tau(b_{j,n})|=0\quad {\rm for\ }i\neq j .\]
Then it follows that 
\begin{align*}
&\lim_{n\to\infty} \max_{\tau\in \deltae(T(A))}| \tau(b_{i,n}b_{j,n})| \\
&=\lim_n \max_{\tau} |\tau(b_{i,n}b_{j,n})-\tau(b_{i,n})\tau(b_{j,n})|+|\tau(b_{i,n})\tau(b_{j,n})-f_i(\tau)f_j(\tau)|=0 \quad {\rm for\ }i\neq j. 
\end{align*}

In the following argument, we will provide the orthogonality of $\{(b_{i,n})_n\}_{i=1}^N$ in the operator norm sense. Let $g_m(t)= \min\{ 1, mt\}$ for $t\in[0, \infty)$ and $m\in\N$. Define central sequences $(r_{i,n})_n$ and $(a_{i,n,m})_m$ by
\begin{align*}
r_{i,n} &= b_{i,n}^{1/2} (\sum_{i\neq j} b_{j,n} )b_{i,n}^{1/2}, \\
a_{i,n,m} &= b_{i,n}^{1/2} (1- g_m (r_{i,n}))b_{i,n}^{1/2}. 
\end{align*}
Note that $a_{i,n,m} \leq b_{i,n}$ for any $i=1,2,...,N$, $n, m\in\N$. Since $\displaystyle \max_{\tau\in \deltae (T(A))} \tau (b_{i,n}b_{j,n}) \to 0$ as $n\to \infty$ for $i\neq j$, it follows that for any $k\in\N$ 
\[ \max_{\tau\in\deltae(T(A))} \tau(r_{i,n}^k) \leq (\sum_{i\neq j} \|f_i\|\cdot\|f_j\|)^{k-1} \cdot \max_{\tau} \tau(\sum_{i\neq j} b_{i,n}b_{j,n}) \to 0,\quad n\to \infty.\]
Then for any $m\in\N$
\[\max_{\tau\in\deltae(T(A))} \tau (b_{i,n} - a_{i,n,m}) \leq \|f_i\|\cdot \max_{\tau} \tau(g_m(r_{i,n}))\to 0, \quad n\to\infty. \]
Furthermore, for $i=1,2,...,N$, $i\neq j$, $n\in\N$ we have
\begin{align*}
\| a_{i,n,m}a_{j,n,m} \|^2 &\leq \|f_j\| \cdot \| a_{i,n,m} a_{j,n,m} a_{i,n,m}\|\leq \|f_j\| \cdot \| a_{i,n,m}( \sum_{i\neq j} b_{j,n}) a_{i,n,m}\| \\
&\leq \|f_i\|\cdot\|f_j\| \cdot\|(1-g_m(r_{i,n})) r_{i,n}\| < \|f_i\|\cdot\|f_j\|/m. 
\end{align*}
Since $(a_{i,n,m})_n\in \Acentral$ for $m\in\N$, we can find an increasing sequence $(m_n)_n$ of natural numbers such that $m_n \to \infty$ as $n\to\infty$, $\displaystyle \max_{\tau\in \deltae(T(A))} \tau (b_{i,n} - a_{i,n,m_n})\to 0$ as $n\to\infty$, and $(a_{i,n, m_n})_n\in \Acentral$. Consequently $a_{i,n}= a_{i,n,m_n}$, $n\in\N$ satisfy the desired conditions.
\end{proof}

\section{Proof of the main theorem}\label{Sec5}
In this section, we give a proof of Theorem \ref{MainThm}. The following lemma and its proof is inspired by techniques for $C(X)$-algebras from \cite[Theorem 4.6]{HRW} and \cite[Theorem 0.1]{DW}. Recently, these techniques were developed by A. S. Toms and W. Winter to show $\Zjs$-absorption of the crossed product C$^*$-algebras by minimal homeomorphisms on compact finite-dimensional spaces \cite{TW}. Their proof as well as ours relies heavily on the condition of finite covering dimenion. It might be interesting that in this lemma the number of completely positive maps corresponds to the covering dimension of $\deltae (T(A))$. 
\begin{proposition}\label{MainLem}
Let $A$ be a unital separable simple nuclear C$^*$-algebra. Suppose that $\deltae (T(A))$ is compact and $d=\dim (\deltae (T(A))) <\infty$. Then for any $k\in\N$ there exist order zero completely positive maps $\varphi_l :M_k \rightarrow \Atcentral$, $l=0,1,...,d$ such that 
\[\sum_{l=0}^d \varphi_l (1_k) =1\quad {\rm and\ }\quad [\varphi_l (a), \varphi_m (b) ] = 0\quad {\rm for\ } l\neq m,\quad a, b\in M_k.\]
\end{proposition}
\begin{proof}
Since $\deltae (T(A))$ is compact and metrizable, we can see that the covering dimension $d$ coincides with the inductive dimension of $\deltae (T(A))$ by \cite[CH.4, Theorem 5.4]{Pea}. This means that for any finite open covering $\{U_i\}_{i=1}^N$ of $\deltae (T(A))$ there exists another open covering $\{V_i\}_{i=1}^N$ such that $\bigcap_{j=1}^{d+2}V_{i_j} =\emptyset$ for $\{ i_j\}_{j=1}^{d+2}\subset \{1,2,...,N\}$, $\overline{V_i}\subset U_i$ for $i=1,2,...,N$, and $\dim (\bd  (V_i))\leq d-1$ for $i=1,2,...,N$, where we denote by $\bd(X)$ the (topological) boundary of a set $X$, \cite[CH.3, Proposition 1.6, and CH.4, Section 2]{Pea}. We shall show the statement by the standard induction for boundaries.

Let $c\in \Z_+$ be such that $c\leq d$. We assume that for any closed subset $B_0\subset \deltae (T(A))$ with $\dim (B_0) < c$ there exist completely positive maps $\psi_{l,n} : M_k \rightarrow A$, $l=0,1,...,c-1$, $n\in\N$ such that 
\begin{enumerate}
\item 
$\displaystyle \left(\sum_{l=0}^{c-1} \psi_{l,n} (1_k)\right)_n \leq 1_{\Acentral}$ \ in $\Acentral$,
\item 
$\displaystyle \lim_{n\to \infty} \lVert \psi_{l,n} (x)\psi_{l,n}(y)\rVert =0\quad {\rm for\ positive\ elements\ } x, y\in {M_k}\ {\rm with\ } x\cdot y=0, $
\item 
$\displaystyle \lim_{n\to \infty} \lVert [\psi_{l,n} (x), a ]\rVert =0\quad {\rm for\ }x\in M_k, a\in A,$
\item
$\displaystyle \lim_{n\to\infty} \lVert [\psi_{l,n} (x), \psi_{m,n}(y)]\rVert =0 \quad{\rm for\ } l\neq m,\ x, y\in M_k,\quad {\rm and } $
\item 
$\displaystyle \lim_{n\to\infty} \max_{\tau\in B_0} \left\lVert 1- \sum_{l=0}^{c-1} \psi_{l,n} (1_k) \right\rVert_{\tau} =0. $
\end{enumerate}  
Recall that both covering dimension and the inductive dimension of $\emptyset$ are defined as $-1$. So, if $d=0$ we regard $B_0=\emptyset$ and this assumption is automotically true. Let $B$ be a closed subset of $\deltae (T(A))$ with $\dim (B)=c$, $F$ a finite subset of contractions in $A$, and $\varepsilon >0$. In order to complete the induction, it suffices to show that: There exist completely positive maps $\varphi_{l,n}:M_k \rightarrow A$, $l=0,1,...,c$, $n\in\N$ such that 
\begin{enumerate}
\item[(i)'] 
$\displaystyle \left(\sum_{l=0}^{c} \varphi_{l,n} (1_k)\right)_n \leq 1_{\Acentral}$, 
\item[(ii)']
$\displaystyle \limsup_{n\to \infty} \lVert \varphi_{l,n} (x)\varphi_{l,n}(y) \rVert<\varepsilon\lVert x\rVert\cdot\lVert y\rVert \quad {\rm for\ positive\ elements\ } x, y\in {M_k}\ {\rm with\ } x\cdot y=0, $
\item[(iii)']
$\displaystyle \limsup_{n\to \infty} \rVert [\varphi_{l,n} (x), a ]\lVert <\varepsilon\lVert x\rVert\quad {\rm for\ }x\in M_k, a\in F,$
\item[(iv)']
$\displaystyle \lim_{n\to\infty} \lVert [\varphi_{l,n} (x), \varphi_{m,n}(y)]\rVert =0 \quad{\rm for\ } l\neq m,\ x, y\in M_k, {\rm and } $
\item[(v)']
$\displaystyle \limsup_{n\to\infty} \max_{\tau\in B} \lVert 1- \sum_{l=0}^{c} \varphi_{l,n} (1_k) \rVert_{\tau} <\varepsilon. $
\end{enumerate}  

Since $A$ is nuclear, for any $\tau\in \deltae (T(A))$ we see that the weak closure of $A$ defined by the GNS-representation associated with $\tau$ is the AFD type II$_1$ factor von Neumann algebra \cite{Con}. Then, by the same argument in the proof of \cite[Lemma 3.3]{MS2} we can obtain a sequence of completely positive contractions $\varphi_{\tau, n} :M_k \rightarrow A$, $n\in\N$ such that $(\varphi_{\tau, n} (a) )_n\in \Acentral $ for $a\in M_k$, $(\varphi_{\tau, n}(x)\varphi_{\tau, n}(y))_n =0$ in $\Acentral$ for positive elements $x, y\in M_k$ with
 $x\cdot y =0$, and $|\tau ( 1- \varphi_{\tau, n} (1_k))| \to 0$ as $n\to \infty$. Thus taking a large $n\in\N$ we obtain completely positive maps $\varphi_{\tau} : M_k \rightarrow A$, $\tau\in \deltae (T(A))$ such that
\begin{align*}
&\|\varphi_{\tau} (x)\varphi_{\tau}(y) \|< \varepsilon \|x\|\cdot\|y\|\quad{\rm for\ any\ positive\ elements\ } x, y\in M_k \ {\rm with\ } x\cdot y =0, \\
&\|[\varphi_{\tau} (x), a ] \| < \varepsilon \|x\|\quad {\rm for\ } x\in M_k, a\in F,\\
&\|1-\varphi_{\tau} (1_k) \|_{\tau} \leq \tau (1- \varphi_{\tau} (1_k))^2 < \varepsilon. 
\end{align*}
For any $\tau\in \deltae (T(A))$, set open subsets 
\[ U_{\tau}=\{\sigma\in \deltae (T(A)) : \|1-\varphi_{\tau} (1_k) \|_{\sigma} < \varepsilon \}. \]
Since $B\subset \deltae (T(A))$ is compact, there exist $\tau_1, \tau_2,...,\tau_N\in \deltae (T(A))$ such that $\displaystyle \bigcup_{i=1}^N U_{\tau_i} \supset B$. Because the inductive dimension of $B$ is equal to $\dim (B) =c$, there exist relatively open covering $V_1, V_2,...,V_N$ of $B$ such that $\overline{V_i} \subset U_i \cap B$, $\displaystyle \bigcap_{j=1}^{c+2} V_{i_j} =\emptyset$ for any $\{i_j\}_{j=1}^{c+2} \subset \{1,2,...,N\}$, and $\dim (\bd_B (V_i))\leq c-1$ for $i=1,2,...,N$. Set $\displaystyle B_0= \bigcup_{i=1}^N \bd_B(V_i)$. Note that $\dim (B_0) \leq c-1$.

By the assumption of the induction we have completely positive contractions $\psi_{l,n}:M_k\rightarrow A$, $l=0,1,...,c-1$, $n\in\N$ satisfying the conditions (i)$\sim$(v) for this $B_0$. Let $\varepsilon_n > 0$, $n\in\N$ be a decreasing sequence such that $\varepsilon_n \to 0$ and $\displaystyle \max_{\tau\in B_0} \left\lVert1-\sum_{l=0}^{c-1}\psi_{l,n}(1_k) \right\rVert_{\tau} < \varepsilon_n$ for $n\in\N$. Let $W_{0,n}'$, $n\in\N$ be relatively open subsets of $B$ such that $W_{0,n}'\supset B_0 $ and $\displaystyle \sup_{\tau\in W_{0,n}'} \left\lVert 1- \sum_{l=0}^{c-1} \psi_{i,n}(1_k)\right\rVert_{\tau} < \varepsilon_n$.  Define $\displaystyle W_i'= V_i\setminus \left(\bigcup_{j=1}^{i-1} \overline{V_j} \cup B_0\right)$, $i=1,2,...,N$ inductively, then it follows that $\displaystyle \bigcup_{i=1}^N W_i' =\bigcup_{i=1}^N V_i\setminus B_0$ and $W_i'\cap W_j'=\emptyset$ for $i\neq j$. Note that $\{ W_{0,n}'\}\cup \{W_i'\}_{i=1}^N $ is a relatively open covering of $B$ for any $n\in\N$. It is elementary to see that there exist open subsets $W_{0,n}$ and $W_i$ of $\deltae (T(A))$ such that 
\[ W_{0,n}\cap B = W_{0,n}',\quad W_i\cap B=W_i',\quad W_i\cap W_j =\emptyset\quad {\rm for\ } i\neq j,\quad {\rm and\ }\bigcup_{i=1}^N W_i\cup W_{0,n}\supset \deltae (T(A)).\] 
(Actually, in general normal space, it is well-known that every pair of separated F$_\sigma$ sets can be devided by disjoint open sets, (see \cite[Problem 2.7.2]{Engel} for example). Now $B\subset \deltae (T(A))$ is a metric space, then each $\{ W_i'\}_{i=1}^N$ is a F$_\sigma$ set in $\deltae (T(A))$ and mutually separated, i.e., $\overline{W_i'}\cap W_j =\emptyset$ for $i\neq j$. Thus there exist open sets $Y_i$ in $\deltae (T(A))$ such that $Y_i\cap Y_j=\emptyset$, $i\neq j$, and $Y_i \supset W_i'$. It is easy to see $W_i =Y_i\setminus (B\setminus W_i')$ are open subsets of $\deltae (T(A))$ such that $W_i \cap B= W_i'$.)
For any $n\in\N$ there exist a partition of unity $\{ f_{0,n}\} \cup\{f_{i,n}\}_{i=1}^N \subset C(\deltae (T(A)))$ such that $\supp (f_{0,n})\subset W_{0,n}$ and $\supp (f_{i,n})\subset W_i$.

Applying (ii) of Lemma \ref{TechLem} to $\{f_{i,n}\}_{i=1}^N$ for each $n\in\N$, we obtain central sequences $(a_{i,n,m})_m$, $i=1,2,...,N$ of positive contractions in $A$ such that 
\[ \lim_{m\to\infty} \max_{\tau\in \deltae (T(A))} |\tau(a_{i,n,m})- f_{i,n}(\tau)| =0,\quad \lim_{m\to\infty} \| a_{i,n,m}a_{j,n,m}\| =0\quad {\rm for\ } i\neq j,\ n\in\N.\] Let $a_{0,n,m}$, $n,m\in\N$ be sequences of positive contractions in $A$ such that 
\[ (a_{0,n,m})_m = \left(1- \sum_{i=1}^N a_{i,n,m}\right)_m\quad {\rm in\ } \Acentral\ {\rm for\ any\ } n\in\N .\]
Note that $([a_{0,n,m}, a_{i,n,m}])_m=0$ in $\Acentral$ for $i=1,2,...,N$. From $\displaystyle f_{0,n}= 1-\sum_{i=1}^N f_{i,n}$ it follows that
\[ \lim_{m\to \infty} \max_{\tau\in\deltae (T(A))} |\tau (a_{0,n,m}) - f_{0,n}(\tau)| =0.\]
Then, by the separability of $A$ and by (i) of Lemma \ref{TechLem}, we can take a subsequence $m_n\in \N$, $n\in\N$ such that 
\begin{enumerate}
\item[(a-i)] $\displaystyle (a_{i,n,m_n})_n \in \Acentral\ {\rm for\ }i=0,1,...,N$,
\item[(a-ii)] $\displaystyle (\sum_{i=0}^N a_{i,n,m_n})_n =1$, 
\item[(a-iii)] $\displaystyle \lim_{n\to \infty} \| a_{i,n,m_n} a_{j,n,m_n} \| =0$  for $i,j\in \{ 1,2,...,N\}$ with $i\neq j$,
\item[(a-iv)] $\displaystyle \lim_{n\to \infty} \|[ a_{0,n,m_n}, a_{i,n,m_n}]\| =0$ for $i=1,2,...,N$,  
\item[(a-v)] $\displaystyle \lim_{n\to \infty} \| [a_{i,n,m_n}, \psi_{j,n} (x) ]\| =0$ for $i,j=0,1,...,N$, $x\in M_k$,
\item[(a-vi)] $\displaystyle \lim_{n\to \infty}\max_{\tau\in \deltae (T(A)) } | \tau (a_{i,n,m_n}) - f_{i,n} (\tau) | = 0$ for $i=0,1,...,N$,
\item[(a-vii)]$\displaystyle \lim_{n\to\infty} \max_{\tau\in \deltae (T(A))} | \tau(a_{0,n,m_n} \psi_{i,n}(x)) - \tau (a_{0,n,m_n})\cdot \tau (\psi_{i,n}(x)) | =0$ for $x\in M_k$. 
\end{enumerate}
Finally, we define completely positive maps $\varphi_{l,n}:M_k \rightarrow A$, $l=0,1,...,c$, $n\in\N$ by 
\begin{align*}
&\varphi_{l,n} (x) = a_{0,n,m_n}^{1/2} \psi_{l,n} (x) a_{0,n,m_n}^{1/2},\quad l=0,1,...,c-1,\ x\in M_k, \\
& \varphi_{c,n} (x) = \sum_{i=1}^N a_{i,n,m_n}^{1/2} \varphi_{\tau_i} (x) a_{i,n,m_n}^{1/2}, \quad x\in M_k.
\end{align*} 
It is straight forward to check that these positive maps satisfy the desired conditions (i)'$\sim$(v)'. 
\end{proof}
\begin{proof}[Proof of Theorem \ref{MainThm}.]
By Lemma \ref{MainLem} and (i) of Corollary \ref{CorDelta} we obtain a unital $*$-homomorphism from $\Delta_{d,k}$ to $\Atcentral$. Since $\Delta_{d,k}$ contains $M_k$ unitally, (ii) of Corollary \ref{CorDelta}, we can conclude the proof.
\end{proof}

{\bf Acknowledgments.}
This work was done during the author's
stay at the University of Oregon. The author would like to thank N. C. Phillips and Kai Wang for their kind help and warm hospitality in Oregon. 
The preliminary work of this paper was done when the author
visited The Research Center for Operator Algebras at East China Normal University in Shanghai, for Operator Algebras Program 2012 Apr. The author also would like  to thank H. Lin for this opportunity and generous hospitality.

This work was supported by the JSPS Research Fellowships for Young Scientists.

\end{document}